\documentclass{amsart}
\usepackage[utf8]{inputenc}

\usepackage{hyperref}

\usepackage{amsmath}
\usepackage{amssymb}
\usepackage{tikz-cd}
\usepackage{enumitem}
\usepackage{bbm}
\usepackage{todonotes}
\usepackage{ifthen}

\theoremstyle{plain}
\newtheorem{thm}{Theorem}[section]
\newtheorem{lemma}[thm]{Lemma}
\newtheorem{prop}[thm]{Proposition}

\newtheorem*{remark}{Remark}
\theoremstyle{definition}
\newtheorem{definition}[thm]{Definition}

\title{Transcendence degrees over mutually generic extensions}

\author{Jonathan Schilhan}
\address{University of Vienna\\
Institute of Mathematics\\
Kurt Gödel Research Center\\
Kolingasse 14-16\\
1090 Vienna\\
Austria}
\email{jonathan.schilhan@univie.ac.at}


\thanks{This research was funded in whole or in part by the Austrian Science Fund (FWF) [10.55776/ESP5711024]. For open access purposes, the authors have applied a CC BY public copyright license to any author-accepted manuscript version arising from this submission. We would like to thank Azul Fatalini for their valuable comments, which improved the presentation of this article.}

\begin{document}

\begin{abstract}
Let $G_0$, \dots, $G_{n-1}$ be mutually generic over $V$, each $G_i$ adding at least one new real over $V$. We show that the transcendence degree of the reals of $V[G_0, \dots, G_{n-1}]$ is maximal (of size continuum) over the field generated by reals coming from models $V[ G_i : i \in a]$, for a proper subset $a$ of $n$. This answers a question of Fatalini and Schindler.
\end{abstract}

\maketitle

\section{Introduction}

The purpose of this note is to answer the main open question posed by Azul Fatalini and Ralf Schindler in \cite{Fatalini2025}. We assume that the reader is familiar with forcing and some basic algebra and analysis. We also make occasional use of descriptive set theory and basic absoluteness theorems for Polish spaces and Borel functions.

Our main result reads as follows:

\begin{thm}\label{thm:main}
    Let $\mathbb{P}_0$, \dots, $\mathbb{P}_{n-1}$ be forcing notions adding new reals and let $G = G_0 \times \dots \times G_{n-1}$ be $\mathbb{P}_0 \times \dots \times \mathbb{P}_{n-1}$-generic over $V$. Then, in $V[G]$, $\mathbb{R}$ has transcendence degree $\mathfrak{c}$ over the field generated by $\bigcup_{a \subsetneq n} \mathbb{R}^{V[G_i : i \in a]}$.
\end{thm}

In \cite{Fatalini2025}, the authors prove this theorem in the particular case where each $\mathbb{P}_i$ is Cohen forcing and they ask whether it holds true more generally (see \cite[Question~3.15]{Fatalini2025}). The question is originally motivated by \cite[Question 4.2]{Kanovei2021}, which considers Jensen reals. Note that in the case $n = 1$, the theorem simply states that the transcendence degree of (the reals of) a forcing extension (with new reals) over its ground model is maximal. This is relatively easy to see for Cohen forcing from the fact that a single Cohen real already adds a perfect set of mutually generic Cohen reals, and thus an algebraically independent set of size $\mathfrak{c}$ over the ground model. On the other hand, this is less immediate for other types of reals. Nevertheless, an argument for $n = 1$, attributed to B. De Bondt is provided in \cite{Fatalini2025}, which uses a generalization of real algrebraic closure named \emph{really closure} (see Definition~\ref{def:reallyclose} below). This trick is useful more generally and will constitute the outermost layer of our proof for the general case, just as in \cite{Fatalini2025}.

The main non-trivial case is $n = 2$ though and we will mainly focus on that one in our paper. The general case has an analogous proof but it is notationally more heavy and the notions that have to be introduced are less intuitive at first. We will point out the simple modifications that have to be made to reach Theorem~\ref{thm:main} in the last section. Thus we are first and foremost occupied with proving the following:

\begin{thm}\label{thm:2gen}
    Let $\mathbb{P}$ and $\mathbb{Q}$ be forcing notions adding new reals and let $G \times H$ be $\mathbb{P} \times \mathbb{Q}$-generic over $V$. Then, in $V[G\times H]$, $\mathbb{R}$ has transcendence degree $\mathfrak{c}$ over the field generated by $\mathbb{R}^{V[G]} \cup \mathbb{R}^{V[H]}$.
\end{thm}

Let us remark that something specific must be used about algebraic closure for the theorem to hold. It is well-known for instance, that if $s_0$ and $s_1$ are mutually Sacks generic over $V$, then every real of $V[s_0, s_1]$ is computable from a pair $(t_0, t_1)$, $t_0 \in V[s_0]$, $t_1 \in V[s_1]$. There are essentially two analytic facts that are being used that apply to algebraic closure but not to computable closure. The first, given by Lemma~\ref{lem:analyticclosure}, is that the algebraic closure of a set (in fact the \emph{really closure}, see Definition~\ref{def:reallyclose}) can be obtained by closing it under countably many $C^1$ functions, something which follows from the Implicit Function Theorem. The second well-known fact, Lemma~\ref{lem:Sard}, is basically saying that $C^1$ functions cannot increase dimension. 

The combinatorial core of our argument can be applied fairly generally and is related to the canonization of two-variable functions on perfect rectangles (e.g. in a sense similar to \cite{KanoveiSabokZapletal}), see Definition~\ref{def:similar}. The following can be seen as a first approach to Theorem~\ref{thm:2gen} and is what led us on the path to resolving the question in the first place:

\begin{thm}\label{thm:multadd}
    Let $G \times H$ be $\mathbb{P} \times \mathbb{Q}$-generic and $r$, $s$ new real numbers in $V[G]$ and $V[H]$ respectively. Then $r \cdot s$ is not of the form $u + v$ for any $u \in V[G]$, $v \in V[H]$. 
\end{thm}

The paper is organised as follows. In the next section, we cover some of the less set-theoretic preliminaries needed, for instance the analytic facts stated above. In Section~3 we prove Theorem~\ref{thm:2gen} and ~\ref{thm:multadd} and in Section~4 we generalise the results to obtain Theorem~\ref{thm:main}.

\section{Preliminaries}

First, we recall De Bondt's trick to get transcendence degree $\mathfrak{c}$. Recall that $\mathbb{R}_{>0}$ denotes the open set of positive real numbers.

\begin{definition}\label{def:reallyclose}
    We say that $p \colon \mathbb{R}_{>0} \to \mathbb{R}$ is a \emph{generalised polynomial in the coefficients} $a_i \neq 0, b_i \neq b_j \in \mathbb{R}$, $i,j < n$, $n \geq 1$, if it is of the form $$p(x) = \sum_{i < n} a_i x^{b_i}.$$
    
    We say that a subfield $F$ of $\mathbb{R}$ is \emph{really closed} if it is closed under the zeros of generalized polynomials in coefficients in $F$. To be more precise, whenever $a_i, b_i \in F$ are as above and $\sum_{i < n} a_i x^{b_i} = 0$, then $x \in F$. The \emph{really closure} of a set $S \subseteq \mathbb{R}$ is the smallest really closed subfield of $\mathbb{R}$ containing $S$.
\end{definition}

\begin{lemma}[{De Bondt, see \cite[Proposition 2.4]{Fatalini2025}}]\label{lem:really}
    Let $F \subseteq \mathbb{R}$ be really closed. If there is $x \in \mathbb{R} \setminus F$, i.e. $F$ is a proper subfield of $\mathbb{R}$, then the transcendence degree of $\mathbb{R}$ over $F$ is $\mathfrak{c}$.
\end{lemma}

The following somewhat implicitly also appears in \cite{Fatalini2025}. Recall that $C^1$ denotes the class of continuously differentiable functions.

\begin{lemma}\label{lem:analyticclosure}
    There is a countable set $\mathcal{F}$ of $C^1$ functions $f$ of the form $f \colon U \to \mathbb{R}$, $U$ an open subset of $\mathbb{R}^n$ for some $n$, so that for any non-empty $S \subseteq \mathbb{R}$, the really closure of $S$ is equal to $\{f(s_0, \dots, s_{k}) : f \in \mathcal{F}, s_0, \dots, s_{k} \in S \}$. 
    
    Moreover, $\mathcal{F}$ has this property in any forcing extension (reinterpreting the functions in the usual way).
\end{lemma}

\begin{proof}
    First, we note that a field is simply generated by recursively applying the functions $(x,y) \mapsto x\cdot y$, $(x,y) \mapsto x + y$, $x \mapsto -x$, $x \mapsto 1$ and $x \mapsto x^{-1}$, the last one of which is defined on the open set $\mathbb{R} \setminus \{ 0\}$. 
    
    To obtain the closure under zeros of generalized polynomials, we use the Implicit Function Theorem. To do so, we first observe that if $a$ is a zero of a generalized polynomial in coefficients in a field $F$, then it is also the zero of such a polynomial $p$, where the derivative of $p$ at $a$, $p'(a)$, is not $0$ (see \cite[Proposition~3.6]{Fatalini2025}). Now consider the general form of such a polynomial, with the coefficients as explicit variables: \begin{equation}\label{eq:genpol}
        p(x, y_0, \dots, y_{n-1}, z_0, \dots, z_{n-1}) = \sum_{i < n} y_i x^{z_i}.
    \end{equation}
    We regard $p$ as a function $\mathbb{R} \times U \to \mathbb{R}$, where $U \subseteq \mathbb{R}^{2n}$ is the open set of $(y_0, \dots, y_{n-1}, z_0, \dots, z_{n-1}) \in \mathbb{R}^{2n}$, such that $y_i \neq 0$, $z_i \neq z_j$ for $i, j < n$. Given any particular $s = (a, a_0, \dots, a_{n-1}, b_0, \dots, b_{n-1}) \in \mathbb{R} \times U$ such that $p(s) = 0$ and $\frac{\partial p}{\partial x}(s) \neq 0$, the Implicit Function Theorem (see e.g. \cite[Theorem~9.28]{Rudin}) guarantees that there are open sets $V \subseteq \mathbb{R}$ and $W \subseteq U$, with $s \in V \times W$, and a continuously differentiable $f \colon W \to \mathbb{R}$, so that for each $t \in W$, $f(t)$ is the unique element of $V$ such that $p(f(t), t) = 0$. In particular, the function $f$ is uniquely determined by $W$ and $V$ and does not depend on $s$. Using the separability of the spaces $\mathbb{R}^k$, we find a countable collection of such open sets $W \times V$ covering the set of all $s$ as above. The number of expressions for generalised polynomials as in (\ref{eq:genpol}) above is also only countable. Followingly, there is a countable collection of $C^1$ functions, so that each zero of a generalised polynomial is obtained by applying one such function to the coefficients. 

    The collection $\mathcal{F}$ claimed by the lemma then simply consists of all the possible compositions of functions just described. The fact that $\mathcal{F}$ works in any transitive model of set theory is a simple absoluteness result. 
\end{proof}

In fact, the functions in $\mathcal{F}$ can even be assumed to be real analytic, by a more general form of the Implicit Function Theorem, but $C^1$ is all we need.

The next lemma is a well-known analytical fact and follows for instance from the much more general Morse-Sard Theorem \cite{Sard}. We include the nice and short proof of this particular case for the readers convenience.

\begin{lemma}\label{lem:Sard}
    Let $F \colon U \to \mathbb{R}^m$ be $C^1$, where $U$ is an open subset of $\mathbb{R}^n$, for $n < m$. Then $F[U]$ has Lebesgue measure $0$ in $\mathbb{R}^m$.
\end{lemma}

\begin{proof}
    Let $a \in \mathbb{R}^{m-n}$ be arbitrary and define $\tilde F \colon U \times \{ a \} \to \mathbb{R}^m$ by $\tilde F(x,a) = F(x)$. $U \times \{a\}$ can be written as the union of countably many compact sets $K_i$, $i \in \omega$. Fix some $i \in \omega$. Since $\tilde F$ is $C^1$, there is a constant $M > 0$ so that $$\| \tilde F(x) - \tilde F(y)\| \leq M \| x-y\|,$$ for all $x,y \in K_i$. If $B$ is an open ball of radius $r$, it's volume is $cr^m$, for a constant $c$ only depending on $m$. Then $\tilde F[B \cap K_i]$ is included in a ball of radius $Mr$ and thus of volume $cM^mr^m$. Now note that $U \times \{ a\}$, and in particular $K_i$, is a null subset of $\mathbb{R}^m$. For any given $\varepsilon > 0$, there is a sequence of open balls covering $K_i$, with total measure $< \varepsilon/M^m$. It follows that $\tilde F[K_i]$ has measure $< \varepsilon$. Since $\varepsilon$ was arbitrary, $\tilde F[K_i]$, and further, $F[U] = \tilde F[U \times \{a\}] = \bigcup \tilde F[K_i]$, has measure $0$.
\end{proof}

\section{Two generics}

The following is the main combinatorial notion we use to obtain an impossibility result such as Theorem~\ref{thm:multadd}.

\begin{definition}\label{def:similar}
    Let $X_0$, $X_1$, $Y_0$, $Y_1$ and $Z$ be Polish spaces, $g\colon X_0 \times X_1 \to Z$ and $f \colon Y_0 \times Y_1 \to Z$. Then we say that $f$ and $g$ are \emph{similar on a perfect rectangle}, if there are non-empty perfect subsets $P_0 \subseteq X_0$, $P_1 \subseteq X_1$ and continuous injections $\varphi_0 \colon P_0 \to Y_0$, $\varphi_1 \colon P_1 \to Y_1$, so that for any $x \in P_0$ and $y \in P_1$, $$g(x,y) = f(\varphi_0(x), \varphi_1(y)).$$

    In other words, we may say that $f$ and $g$ are equal up to substitution of variables on some perfect rectangle. 
\end{definition}

It is worth noting that for Borel functions $f$ and $g$ as above, the property of being similar on a perfect rectangle can be expressed by a $\Sigma^1_2$ sentence and thus is absolute between transitive models of ZFC with the same ordinals.

\begin{prop}\label{prop:mainpair}
    Let $G \times H$ be $\mathbb{P} \times \mathbb{Q}$-generic over $V$. Let $g \colon X_0 \times X_1 \to Z$ and $f \colon Y_0 \times Y_1 \to Z$ as above be Borel functions in $V$, not similar on a perfect rectangle. Further, let $r \in (X_0)^{V[G]} \setminus V$ and $s \in (X_1)^{V[H]} \setminus V$. Then $g(r,s)$ is not of the form $f(u,v)$ for any $u \in (Y_0)^{V[G]} \setminus V$ and $v \in (Y_1)^{V[H]} \setminus V$.
\end{prop}

\begin{proof}
    Let $\dot r$, $\dot u$ be $\mathbb{P}$-names and $\dot s$, $\dot v$ be $\mathbb{Q}$-names for $r$, $u$, $s$ and $v$ respectively. Suppose towards a contradiction, that there is $(p, q) \in \mathbb{P} \times \mathbb{Q}$, such that $$(p, q) \Vdash \dot r, \dot s, \dot u, \dot v \notin V \wedge g(\dot r, \dot s) = f(\dot u, \dot v),$$ where the names are viewed as $\mathbb{P} \times \mathbb{Q}$-names in the natural way. 

    Let $M$ a countable elementary submodel of some large fragment of the universe, reflecting the above and containing all relevant parameters. Then we can find a ``perfect rectangle" of generics over $M$, witnessing that $g$ and $f$ are similar. 
    
    To be more precise, consider the forcing notion in $M$ consisting of pairs $(h_0, h_1)$, where $h_0 \colon 2^{< n} \to \{ p' \leq p : p' \in \mathbb{P}\}$ and $h_1 \colon 2^{<n} \to \{ q' \leq q : q' \in \mathbb{Q}\}$, for some $n \in \omega$, and $\eta \subseteq \nu \in 2^{<n}$ implies $h_i(\nu) \leq h_i(\eta)$, for $i \in 2$. The order is given by $(h_0', h_1') \leq (h_0, h_1)$ iff $h_0 \subseteq h_0'$ and $h_1 \subseteq h_1'$. A generic over $M$ induces reverse-order-preserving functions $\chi_0 \colon 2^{< \omega} \to \mathbb{P}$ and $\chi_1 \colon 2^{< \omega} \to \mathbb{Q}$. Then the following properties are straightforward genericity arguments:

    \begin{enumerate}
        \item For every $x,y \in 2^\omega$, $\{ \chi_0(x \restriction n) : n \in \omega \} \subseteq \mathbb{P}$ and $\{ \chi_1(y \restriction n) : n \in \omega \} \subseteq \mathbb{Q}$ generate filters $G_x$ and $H_y$, respectively, with $G_x \times H_y$ being $\mathbb{P} \times \mathbb{Q}$-generic over $M$.
        \item For $\tau \in \{ \dot r, \dot s, \dot u, \dot v \}$, the map $e_\tau$ with domain $2^\omega$, given by $$e_\tau(x) =\begin{cases}
             \tau^{G_x} &\text{ if } \tau \in \{\dot r, \dot u \}\\
            \tau^{H_x} &\text{ if } \tau \in \{\dot s, \dot v \}
        \end{cases}$$ is a continuous injection from $2^\omega$ into $X_0$, $X_1$, $Y_0$ or $Y_1$ respectively.
        \item For each $x, y \in 2^\omega$, $$g(e_{\dot r}(x), e_{\dot s}(y)) = f(e_{\dot u}(x), e_{\dot v}(y)).$$
    \end{enumerate}

    Now it suffices to note that $P_0 := e_{\dot r}[2^\omega] \subseteq X_0$ and $P_1 := e_{\dot s}[2^\omega] \subseteq X_1$ are perfect sets, being the injective and continuous image of $2^\omega$, and that $\varphi_0 := e_{\dot u} \circ e_{\dot r}^{-1}$ and $\varphi_1 := e_{\dot v} \circ e_{\dot s}^{-1}$ are continuous injections. 
\end{proof}

In fact, in the argument we make, a much weaker notion than that of Definition~\ref{def:similar} will be sufficient and it is worth pointing this out. 

\begin{definition}
    Let $f$ and $g$ be as above, $n \leq \omega$. Then $f$ and $g$ are said to be \emph{similar on an $n \times n$ grid} if there are $n$-sized sets $P_0 \subseteq X_0$, $P_1 \subseteq X_1$ and injections $\varphi_0 \colon P_0 \to Y_0$, $\varphi_1 \colon P_1 \to Y_1$, so that for any $x \in P_0$ and $y \in P_1$, $$g(x,y) = f(\varphi_0(x), \varphi_1(y)).$$
\end{definition}

Obviously, being similar on a perfect rectangle is much stronger then being similar on an $n \times n$ grid, for $n \leq \omega$. The following simple observation leads to Theorem~\ref{thm:multadd}.

\begin{lemma}\label{lem:multadd}
    Multiplication is not similar to addition on a $2 \times 2$ grid  (as maps from $\mathbb{R}\times \mathbb{R}$ to $\mathbb{R}$). 
\end{lemma}

\begin{proof}
Suppose towards a contradiction that there are $P_0 = \{a, b \}$, $P_1 = \{ c, d \}$ and injections $\varphi_0$, $\varphi_1$, such that $x + y = \varphi_0(x) \cdot \varphi_1(y)$, for $x \in P_0$, $y \in P_1$. Then we have that $(a+c)-(a+d)=c-d=(b+c)-(b+d)$. Thus also \begin{align*}
    \varphi_0(a)(\varphi_1(c) -\varphi_1(d)) &= \varphi_0(a) \cdot \varphi_1(c) - \varphi_0(a) \cdot \varphi_1(d) \\ &= \varphi_0(b) \cdot \varphi_1(c) - \varphi_0(a) \cdot \varphi_1(d) \\ &= \varphi_0(b)(\varphi_1(c) -\varphi_1(d)).
\end{align*}
Since $\varphi_1$ is injective, $\varphi_1(c) -\varphi_1(d) \neq 0$. Hence, we obtain that $\varphi_0(a) = \varphi_0(b)$. But this contradicts the injectivity of $\varphi_0$.
\end{proof}

\begin{lemma}\label{lem:unsimilar}
    Let $\mathcal{F}$ be a countable set of $C^1$ functions of the form $f \colon U \times V \to \mathbb{R}$, for open sets $U, V \subseteq \mathbb{R}^n$, $n \geq 1$. Then there is a Borel function $g \colon \mathbb{R} \times \mathbb{R} \to \mathbb{R}$ not similar to any $f \in \mathcal{F}$ on an $\omega \times \omega$ grid.
\end{lemma}

\begin{proof}
Let $M$ be a countable elementary submodel of some sufficiently large fragment of $V$ containing all of $\mathcal{F}$, and let $P \subseteq \mathbb{R}$ be a perfect set of mutually generic random reals over $M$, i.e. for every $k \in \omega$ and every $\bar z \in P^k$, $\bar z$ is not a member of any Lebesgue null subset of $\mathbb{R}^k$ in $M$. Let $g \colon \mathbb{R} \times \mathbb{R} \to \mathbb{R}$ be any Borel injection into $P$. Then we claim that $g$ works. 

Towards this end, let $f \colon U \times V \to \mathbb{R}$ be in $\mathcal{F}$, where $U, V \subseteq \mathbb{R}^n$. Let $m \in \omega$ be large enough so that $2nm < m^2$ (e.g. $m = 2n + 1$). Consider the map $F \colon U^m \times V^m \to \mathbb{R}^{m^2}$, where for $i, j <m$, $$F(x_0, \dots, x_{m-1}, y_0, \dots, y_{m-1})_{i,j} = f(x_i, y_j).$$

$F$ is a $C^1$ map from an open subset of $\mathbb{R}^{2nm}$ to $\mathbb{R}^{m^2}$. By Lemma~\ref{lem:Sard}, $F[U^m \times V^m]$ has measure $0$ in $\mathbb{R}^{m^2}$ and it is an element of $M$, by elementarity. It readily follows that $g$ cannot be similar to $f$ on an $m \times m$-grid, since the set of values of $g$ on any set of the form $P_0 \times P_1$, with $\vert P_0\vert = \vert P_1 \vert = m$, disagrees with the set of values of $f$ on such a set.
\end{proof}

The proof of Theorem~\ref{thm:2gen} is easily obtained with the help of Lemma~\ref{lem:analyticclosure} and Lemma~\ref{lem:really}.

\begin{proof}[Proof of Theorem~\ref{thm:2gen}]
    Using Lemma~\ref{lem:analyticclosure}, there is a countable set of $C^1$ functions $\mathcal{F} \in V$, so that each element of the really closure of $\mathbb{R}^{V[G]} \cup \mathbb{R}^{V[H]}$ is of the form $f(u,v)$ for some $f \in \mathcal{F}$, $u \in (\mathbb{R}^n)^{V[G]} \setminus V$ and $v \in (\mathbb{R}^n)^{V[H]} \setminus V$, $n \geq 1$.\footnote{Note specifically, that we can add ``dummy variables" to any given $f \in \mathcal{F}$, in order to write every element precisely in this form. For instance, if $z \in V$ is used to generate an element of the really closure, then we can use $u = (z,r) \in V[G] \setminus V$, for some $z \in V[G] \setminus V$.} Let $g$ be as in the previous lemma, $r$ a new real of $V[G]$ and $s$ a new real of $V[H]$. By Proposition~\ref{prop:mainpair}, $g(r,s)$ is not of the form just described, and followingly is not a member of the really closure of $\mathbb{R}^{V[G]} \cup \mathbb{R}^{V[H]}$. Lemma~\ref{lem:really} finishes the argument.
\end{proof}

\section{Any number of generics}

The following is the ad-hoc generalisation of Definition~\ref{def:similar} that applies to our situation.

\begin{definition}
    Let $g \colon X^n \to Z$ and $f \colon \prod_{a \subsetneq n} Y_a \to Z$, for some $n \geq 2$, where $X$, $Y_a$ and $Z$ are Polish spaces. Then we say that $f$ and $g$ are \emph{similar}, if there are non-empty perfect sets $P_0, \dots, P_{n-1} \subseteq X$ and continuous injections $\varphi_a \colon \prod_{i \in a} P_i \to Y_a$, for each $a \subsetneq n$, such that for every $x_0 \in P_0, \dots, x_{n-1} \in P_{n-1}$, $$g(x_0, \dots, x_{n-1}) = f(\langle \varphi_a(\langle x_i : i \in a \rangle) : a \subsetneq n \rangle).$$
\end{definition}

The analogue of Proposition~\ref{prop:mainpair} is proved in essentially the same way and we leave the details to the reader. 

\begin{prop}
    Assume that $f$ and $g$ as above are not similar. Let $r_i \in X^{V[G_i]} \setminus V$, for each $i < n$. Then $g(r_0, \dots, r_{n-1})$ is not of the form $f( \langle u_a : a \subsetneq n \rangle)$, where $u_a \in Y_a^{V[ G_i : i \in a ]} \setminus \bigcup_{b \subsetneq a} V[G_i : i \in b ]$, for each $a \subsetneq n$. 
\end{prop}

\begin{lemma}
        Let $n\geq 2$ and $\mathcal{F}$ be a countable set of $C^1$ functions of the form $f \colon \prod_{a \subsetneq n} U_a \to \mathbb{R}$, for open sets $U_a \subseteq \mathbb{R}^k$, $k \geq 1$. Then there is a Borel function $g \colon \mathbb{R}^n \to \mathbb{R}$ not similar to any $f \in \mathcal{F}$.
\end{lemma}

\begin{proof}
    The argument is analogous to the proof of Lemma~\ref{lem:unsimilar}. Specifically, for $f$ as above, we let $m$ be large enough so that $\sum_{a \subsetneq n} k m^{\vert a \vert} < m^n$. Such $m$ exists since the highest exponent in the polynomial (in variable $m$) on the left-hand-side is $n-1$, while $m^n$ has exponent $n$. Then consider $F \colon \prod_{a \subsetneq n} (U_a)^{(m^{\vert a\vert})} \to \mathbb{R}^{(m^n)}$ given by $$F(\langle x_{a, e} : e \in {}^{\vert a \vert}m, a \subsetneq n \rangle)_h = f( \langle x_a,h \restriction a : a \subsetneq n \rangle),$$ where we think of elements of $\mathbb{R}^{(m^n)}$ as indexed by functions $h \colon n \to m$. The image of $F$ has measure zero in $\mathbb{R}^{(m^n)}$ and the rest of the argument is as in Lemma~\ref{lem:unsimilar}.
\end{proof} 

Theorem~\ref{thm:main} is obtained exactly as before.

\begin{remark}
    It is likely that a finer non-similarity argument can provide an explicit function $g$ so that $g(r_0, \dots, r_{n-1})$, for $r_i$ a new real in $V[G_i]$, is not in the real algebraic closure of the subfield generated by the intermediate models. Given the type of argument in Lemma~\ref{lem:multadd}, a natural contender would be $e^{r_0 \cdot \dots\cdot r_{n-1}}$.
\end{remark}

\bibliographystyle{plain}

\end{document}